\documentclass[11pt]{article}

\usepackage{amssymb, amsmath, amsthm, tikz, xspace,subfigure, graphicx}
\usepackage[left=1in,top=1in,right=1in]{geometry}
\usetikzlibrary{positioning, shapes.geometric, calc, intersections}
\tikzstyle{vertex}=[circle,fill=black,inner sep=2pt]
\tikzstyle{vertrect}=[draw,rectangle,inner sep=2pt]
\tikzstyle{vertdia}=[draw,diamond,inner sep=2pt]

\date{}

\theoremstyle{plain}
      \newtheorem{theorem}{Theorem}[section]
      \newtheorem{lemma}[theorem]{Lemma}
            \newtheorem{claim}[theorem]{Claim}
      
      \newtheorem{observation}[theorem]{Observation}
      \newtheorem{corollary}[theorem]{Corollary}
      
      \newtheorem{conjecture}[theorem]{Conjecture}
\theoremstyle{definition}

\theoremstyle{remark}

\def\twr{\mbox{\rm twr}}

\title{New lower bounds for hypergraph Ramsey numbers}

\author{Dhruv Mubayi\thanks{Department of Mathematics, Statistics, and Computer Science, University of Illinois, Chicago, IL, 60607 USA.  Research partially supported by NSF grant DMS-1300138. Email: {\tt mubayi@uic.edu}} \and Andrew Suk\thanks{Department of Mathematics,  University of California at San Diego, La Jolla, CA, 92093 USA. Supported by NSF grant DMS-1800736, NSF CAREER award DMS-1800746, and an Alfred Sloan Fellowship. Email: {\tt asuk@ucsd.edu}.\newline Mathematics Subject Classification (2010) codes: 05D10, 05D40, 05C65}}

\begin{document}

\maketitle

\begin{abstract}
 The \emph{Ramsey number} $r_k(s,n)$ is the minimum $N$ such that for every red-blue coloring of the $k$-tuples of $\{1,\ldots, N\}$, there are $s$ integers such that every $k$-tuple among them is red, or $n$ integers such that every $k$-tuple among them is blue.  We prove the following new lower bounds for 4-uniform  hypergraph Ramsey numbers:
$$r_4(5,n) > 2^{n^{c\log n}} \qquad \hbox{ and } \qquad
r_4(6,n) > 2^{2^{cn^{1/5}}},$$
where $c$ is an absolute positive constant. This substantially improves the previous best bounds of $2^{n^{c\log\log n}}$ and $2^{n^{c\log n}}$, respectively. Using previously known upper bounds, our result implies that the growth rate of $r_4(6,n)$ is double exponential in a power of $n$.

As a consequence, we obtain similar bounds for the $k$-uniform Ramsey numbers
$r_k(k+1, n)$ and $r_k(k+2, n)$ where the exponent is replaced by an appropriate tower function.
This almost solves the question of determining the tower growth rate for {\emph {all}} classical off-diagonal hypergraph Ramsey numbers, a question first  posed by
Erd\H os and Hajnal in 1972. The only problem that remains is to prove that $r_4(5,n)$ is double exponential in a power of $n$.
\end{abstract}

\section{Introduction}

 A $k$-uniform hypergraph $H$ with vertex set $V$ is a collection of $k$-element subsets of $V$.  We write $K^{(k)}_n$ for the complete $k$-uniform hypergraph on an $n$-element vertex set.  The \emph{Ramsey number} $r_k(s,n)$ is the minimum $N$ such that every red-blue coloring of the edges of  $K^{(k)}_N$  contains a monochromatic red copy of $K_s^{(k)}$ or a monochromatic blue copy of $K^{(k)}_n$.

 \emph{Diagonal} Ramsey numbers refer to the special case when $s = n$, i.e. $r_k(n,n)$, and have been studied extensively over the past 80 years.  Classic results of Erd\H os and Szekeres \cite{ES35} and Erd\H os \cite{E47} imply that $2^{n/2} < r_2(n,n) \leq 2^{2n}$ for every integer $n > 2$.  While small improvements have been made in both the upper and lower bounds for $r_2(n,n)$ (see \cite{S75,C09}), the constant factors in the exponents have not changed over the last 70 years.

Unfortunately, for 3-uniform hypergraphs, our understanding of $r_3(n,n)$ is much less.  Results of Erd\H os, Hajnal, and Rado \cite{EHR} gives the best known lower and upper bounds for $r_3(n,n)$, $$2^{c_1n^2}<r_3(n,n)<2^{2^{c_2n}},$$ where $c_1$ and $c_2$ are positive constants.  For $k \geq 4$, there is also a difference of one exponential between the known lower and upper bounds for $r_k(n,n)$, that is, \begin{equation}\label{diag}\twr_{k-1}(c_1n^2) \leq r_k(n,n) \leq \twr_k(c_2n),\end{equation}

\noindent where the tower function $\twr_k(x)$ is defined by $\twr_1(x) = x$ and $\twr_{i + 1}(x) = 2^{\twr_i(x)}$ (see \cite{ES35,ER,EH72}).  A notoriously difficult conjecture of Erd\H os, Hajnal, and Rado states that the upper bound in (\ref{diag}) is essentially the truth, that is, there are constructions which demonstrates that $r_k(n,n) > \twr_k(cn)$, where $c = c(k)$.  The crucial case is when $k = 3$, since a double exponential lower bound for $r_3(n,n)$ would verify the conjecture for all $k\geq 4$ by using the well-known stepping-up lemma of Erd\H os and Hajnal (see \cite{graham}).

\begin{conjecture}[Erd\H os]\label{3conj}
For $n\geq 4$, $r_3(n,n) > 2^{2^{cn}}$ where $c$ is an absolute constant.

\end{conjecture}

 \emph{Off-diagonal} Ramsey numbers, $r_k(s,n)$, refer to the special case when $k,s$ are fixed and $n$ tends to infinity.  It is known \cite{AKS,Kim,B,BK} that $r_2(3,n) =\Theta(n^2/\log n)$, and more generally for fixed $s > 3$, $r_2(s,n) = n^{\Theta(1)}$.  For 3-uniform hypergraphs, a result of Conlon, Fox, and Sudakov \cite{CFS} shows that $$2^{c_1n\log n} \leq r_3(s,n) \leq 2^{c_2n^{s-2}\log n},$$  where $c_1$ and $c_2$ depend only on $s$.  For $k$-uniform hypergraphs, where $k\geq 4$, it is known that
$r_k(s,n) \leq \twr_{k-1}(n^{c}),$ where $c = c(s)$ \cite{ER}.  By applying the Erd\H os-Hajnal stepping up lemma in the off-diagonal setting, it follows that \begin{equation}\label{off}r_k(s,n) \geq \twr_{k-1}(c'n),\end{equation} for $k\geq 4$ and $s \geq 2^{k-1} - k + 3$, where $c' = c'(s)$.  In 1972, Erd\H os and Hajnal~\cite{EH72} conjectured that (\ref{off}) holds for every fixed $k\geq 4$ and $s \geq k + 1$.  Actually, this was part of a more general conjecture that they posed in that paper (see \cite{MS15, MS16} for details). In \cite{CFS13}, Conlon, Fox, and Sudakov verified the Erd\H os-Hajnal conjecture for all $s\geq \lceil 5k/2\rceil - 3$.   Very recently, the current authors \cite{MS15} and independently Conlon, Fox, and Sudakov \cite{CFS15} verified the conjecture for all $s \geq k + 3$ (using different constructions).
Since
$2^{k-1} - k + 3=\lceil 5k/2\rceil - 3=k+3=7$ when $k=4$, all three of these approaches succeed in proving a double exponential lower bound for $r_4(7,n)$
but fail for $r_4(6,n)$ and $r_4(5,n)$.
Just as for diagonal Ramsey numbers, a double exponential in $n^c$ lower bound for $r_4(5,n)$ and $r_4(6,n)$ would imply $r_k(k + 1,n) > \twr_{k-1}(n^{c'})$ and $r_k(k + 2,n) > \twr_{k-1}(n^{c'})$ respectively, for all fixed $k\geq 5$.  This follows from a variant of the stepping-up lemma that we will describe in Section~2.
Therefore, the difficulty in verifying (\ref{off}) for the two remaining cases, $s = k+1$ and $k+2$, is due to our lack of understanding of $r_4(5,n)$ and $r_4(6,n)$. Consequently, showing double exponential lower bounds for $r_4(5,n)$ and $r_4(6,n)$ are the only two problems that remain to determine the tower growth rate for all off-diagonal hypergraph Ramsey numbers.

Until very recently, the only lower bound for both $r_4(5,n)$ and $r_4(6,n)$ was  $2^{cn}$,  which was implicit in the paper of Erd\H os and Hajnal~\cite{EH72}. Our results in~\cite{MS15c,MS15} improved both these bounds to
\begin{equation} \label{best}  r_4(5,n)>2^{n^{c\log\log n}} \qquad \hbox{  and } \qquad r_4(6,n) >2^{n^{c\log n}}\end{equation} and these are the current best known bounds. As mentioned above, the bounds in (\ref{best}) imply the corresponding improvements to the lower bounds for $r_k(k+1, n)$ and $r_k(k+2, n)$.  In this paper we further substantially improve both lower bounds in (\ref{best}).

\begin{theorem}\label{mainthm}
For all $n\geq 6$,
$$r_4(5,n)> 2^{n^{c\log n}} \hspace{1cm}\textnormal{and}\hspace{1cm}r_4(6,n)> 2^{2^{cn^{1/5}}},$$
where $c > 0$ is an absolute constant.

\end{theorem}
Using the stepping-up lemma (see Section 2) we obtain the following.
\begin{corollary} \label{cors}
For $n > k \geq 5$, there is a $c = c(k) > 0$ such that
$$r_k(k+1,n)> \twr_{k-2}(n^{c\log n}) \hspace{1cm}\textnormal{and}\hspace{1cm}r_k(k+2,n)> \twr_{k-1}(cn^{1/5}).$$
\end{corollary}
A standard argument in Ramsey theory together with results in~\cite{CFS} for 3-uniform hypergraph Ramsey numbers yields the upper bound $r_k(k+2, n) <
\twr_{k-1}(c'n^3\log n)$, so we now know the tower growth rate of $r_k(k+2, n)$.

In \cite{MS15}, we established a connection between diagonal and off-diagonal Ramsey numbers.  In particular, we showed that a solution to Conjecture \ref{3conj} implies a solution to the following conjecture.
\begin{conjecture}\label{offconj}

For $n\geq 5$, there is an absolute constant $c> 0$ such that $r_4(5,n)  > 2^{2^{n^c}}$.

\end{conjecture}

The main idea in our approach is to apply stepping-up starting from a graph to construct a 4-uniform hypergraph, rather than the usual method of going from a 3-uniform hypergraph to a 4-uniform hypergraph. Although this approach was implicitly developed in~\cite{MS15}, here we use it explicitly.

For more related Ramsey-type results for hypergraphs, we refer the interested reader to \cite{MS15,MS15c,MS16}.   All logarithms are in base 2 unless otherwise stated. For the sake of clarity of presentation, we omit floor and ceiling signs whenever they are not crucial.

\section{The stepping-up lemma and proof of Lemma \ref{step}}

The proof of our main result, Theorem \ref{mainthm}, follows by applying a variant of the classic Erd\H os-Hajnal stepping-up lemma.  In this section, we describe the stepping-up procedure and sketch the proof of Lemma~\ref{step} below which is used to prove Corollary~\ref{cors}. The particular case below can be found in~\cite{MS15c}, though a special case of Lemma~\ref{step}  was communicated to us independently by Conlon, Fox, and Sudakov~\cite{CFS15}.

  \begin{lemma}\label{step}
  For $k\geq 5$ and $n \geq s \geq k + 1$, $r_k(s, 2kn) > 2^{r_{k-1}(s-1,n) - 1}$.

  \end{lemma}

\proof Let $k \geq 5$, $n \geq s\geq k + 1$, and set $A  = \{0,1,\ldots, N-1\}$ where $N = r_{k-1}(s-1,n) - 1$.  Let $\phi:{A\choose k-1} \rightarrow \{\textnormal{red, blue}\}$ be a red/blue coloring of the $(k-1)$-tuples of $A$ such that there is no monochromatic red copy of $K_{s-1}^{(k-1)}$ and no monochromatic blue copy of $K_{n}^{(k-1)}$.  We know $\phi$ exists by the definition of $N$.  Set $V = \{0,1,\ldots, 2^N - 1\}$.  In what follows, we will use $\phi$ to define a red/blue coloring $\chi:{V\choose k}\rightarrow \{\textnormal{red, blue}\}$ of the $k$-tuples of $V$ such that $\chi$ does not contain a monochromatic red copy of $K_{s}^{(k)}$, and does not contain a monochromatic blue copy of $K_{2kn}^{(k)}$.

For any $v \in V$, write $v=\sum_{i=0}^{N-1}v(i)2^i$ with $v(i) \in \{0,1\}$ for each $i$. For $u \not = v$, let $\delta(u,v) \in A$ denote the largest $i$ for which $u(i) \not = v(i)$.  Notice that we have the following stepping-up properties (see \cite{graham})

\begin{description}

\item[Property I:] For every triple $u < v < w$, $\delta(u,v) \not = \delta(v,w)$ .

\item[Property II:] For $v_1 < \cdots < v_r$, $\delta(v_1,v_{r}) = \max_{1 \leq j \leq r-1}\delta(v_j,v_{j + 1})$.

\end{description}

We will also use the following two stepping-up properties, which are easy consequences of Properties I and II.

\begin{description}

\item[Property III:] For every 4-tuple $v_1 < \cdots < v_4$, if $\delta(v_1,v_2) > \delta(v_2,v_3)$, then $\delta(v_1,v_2) \neq \delta(v_3,v_4)$.   Note that if $\delta(v_1,v_2) < \delta(v_2,v_3)$, it is possible that $\delta(v_1,v_2) = \delta(v_3,v_4)$.

\item[Property IV:] For $v_1 < \cdots < v_r$, set $\delta_j = \delta(v_j,v_{j + 1})$ and suppose that $\delta_1,\ldots, \delta_{r-1}$ forms a monotone sequence.  Then for every subset of $k$-vertices $v_{i_1},v_{i_2},\ldots, v_{i_k}$, where $v_{i_1} < \cdots < v_{i_k}$, $\delta(v_{i_1},v_{i_2}), \delta(v_{i_2},v_{i_3}),\ldots, \delta(v_{i_{k-1}},v_{i_k})$ forms a monotone sequence.  Moreover, for every subset of $k-1$ vertices $\delta_{j_1},\delta_{j_2},\ldots, \delta_{j_{k-1}}$, there are $k$ vertices $v_{i_1},\ldots, v_{i_k}$ such that $\delta(v_{i_t}, v_{i_{t + 1}}) = \delta_{j_t}$.

\end{description}

Given any $k$-tuple $v_1<v_2<\ldots<v_{k}$ of $V$, consider the integers $\delta_i=\delta(v_i,v_{i+1}), 1\le i\le k-1$. We say that $\delta_i$ is a {\it local minimum} if $\delta_{i-1}>\delta_i<\delta_{i+1}$, a {\it local maximum} if $\delta_{i-1}<\delta_i>\delta_{i+1}$, and a {\it local extremum} if it is either a local minimum or a local maximum.  Since $\delta_{i-1} \not = \delta_i$ for every $i$, every nonmonotone sequence $\delta_1,\ldots,\delta_{k-1}$ has a local extremum.

Using $\phi:{A\choose k-1}\rightarrow \{\textnormal{red, blue}\}$, we define $\chi:{V\choose k}\rightarrow \{\textnormal{red, blue}\}$ as follows.  For $v_1 < \cdots < v_k$ and $\delta_i = \delta(v_i,v_{i + 1})$, we define $\chi(v_1,\ldots, v_k) = $ red if

\begin{enumerate}

\item[(a)] $\delta_1,\ldots, \delta_{k-1}$ forms a monotone sequence and $\phi(\delta_1,\ldots, \delta_{k-1}) = $ red, or if

\item[(b)] $\delta_1,\ldots, \delta_{k-1}$ forms a \emph{zig-zag} sequence such that $\delta_2$ is a local maximum.  In other words, $\delta_1 < \delta_2 > \delta_3 < \delta_4 > \cdots.$

\end{enumerate}

Otherwise $\chi(v_1,\ldots, v_k) = $ blue.

For the sake of contradiction, suppose $\chi$ produces a monochromatic red copy of $K^{(k)}_{s}$ on vertices $v_1 < \cdots < v_{s}$, and let $\delta_i = \delta(v_i,v_{i+1})$.  If $\delta_1,\delta_2,\ldots, \delta_{s-1}$ forms a monotone sequence, then by Property IV, $\phi$ colors every $(k-1)$-tuple in the set $\{\delta_1,\ldots, \delta_{s-1}\}$ red, which is a contradiction.  Let $\delta_{i}$ denote the first local extremum in the sequence $\delta_1,\ldots, \delta_{s-1}$.  It is easy to see that $\delta_{i}$ is a local maximum since otherwise we would get a contradiction.  Suppose $i +k - 1 \leq s$.  If $\delta_{i + 1}$ is not a local extremum, then $\chi(v_{i-1},v_i,v_{i+1},\ldots, v_{i + k - 2}) = $ blue which is a contradiction.  If $\delta_{i + 1}$ is a local extremum, then it must be a local minimum which implies that $\chi(v_i,v_{i+1},\ldots, v_{i + k - 1}) = $ blue, contradiction.  Therefore we can assume that $i + k - 1 > s$, which implies $i\geq 3$ since $s\geq k + 1$.  However, this implies that either $\chi(v_{i-2},v_{i-1},\ldots, v_{i + k - 3}) = $ blue or $\chi(v_{s - k + 1},v_{s - k + 2}, \ldots, v_s) = $ blue, contradiction.  Hence, $\chi$ does not produce a monochromatic red copy of $K^{(k)}_s$ in $V$.

Let $m = 2kn$.  For the sake of contradiction, suppose $\chi$ produces a monochromatic blue copy of $K^{(k)}_{m}$ on vertices $v_1,\ldots, v_m$ and let $\delta_i = \delta(v_i,v_{i+1})$.  By Property IV, there is no $x$ such that $\delta_x,\delta_{x + 1},\ldots, \delta_{x + n-1}$ forms a monotone sequence.  Indeed, otherwise $\phi$ would produce a monochromatic blue copy of $K_{n}^{(k-1)}$ on vertices $\delta_x,\delta_{x + 1},\ldots, \delta_{x + n -1}$.  Therefore, we can set $\delta_{i_1},\ldots, \delta_{i_k}$ to be the first $k$ local minimums in the sequence $\delta_1,\ldots, \delta_{m-1}$.  However, by Property II, $\chi$ colors the first $k$ vertices in the set $\{v_{i_1},v_{i_1 + 1},v_{i_2},v_{i_2 + 1},\ldots, v_{i_k},v_{i_k + 1}\}$ red which is a contradiction.  This completes the proof of Lemma \ref{step}.\qed

\section{A double exponential lower bound for $r_4(6,n)$}\label{clique}

The lower bound for $r_4(6,n)$ follows by applying a variant the Erd\H os-Hajnal stepping up lemma.   We start with the following simple lemma which is a straightforward application of the probabilistic method.

 \begin{lemma}\label{start}
There is an absolute constant $c>0$ such that the following holds.  For every $n\geq 6$, there is a red/blue coloring $\phi$ of the pairs of $\{0,1,\ldots, \lfloor 2^{cn}\rfloor - 1\}$ such that

\begin{enumerate}

\item there are no two disjoint $n$-sets $A,B\subset \{0,1,\ldots, \lfloor 2^{cn}\rfloor - 1\}$, such that $\phi(a,b) =$ red for every $a\in A$ and $b\in B$, or $\phi(a,b) =$ blue for every $a \in A$ and $b \in B$ (i.e., no monochromatic $K_{n,n}$),

\item there is no $n$-set $A\subset \{0,1,\ldots, \lfloor 2^{cn}\rfloor - 1\}$  such that every triple $a_i,a_j,a_k \in A$, where $a_i < a_j < a_k$, avoids the pattern $\phi(a_i,a_j) =\phi (a_j,a_k) = blue$ and $\phi(a_i,a_k) = red$.

\end{enumerate}

 \end{lemma}

\begin{proof}

 Set $N = \lfloor 2^{cn}\rfloor$, where $c$ is a sufficiently small constant that will be determined later.  Consider the red/blue coloring $\phi$ of the pairs (edges) of $\{0,1,\ldots, N-1\}$, where each edge has probability $1/2$ of being a particular color independent of all other edges.  Then the expected number of monochromatic copies of the complete bipartite graph $K_{n,n}$ is at most

 $${N\choose n}^22^{-n^2 + 1} < 1/3,$$

 for $c$ sufficiently small and $n\geq 6$.

 We call a triple $a_i,a_j,a_k \in \{0,1,\ldots, N - 1\}$ \emph{bad} if $a_i < a_j < a_k$ and $\phi(a_i,a_j) =\phi (a_j,a_k) =$ blue and $\phi(a_i,a_k) =$ red.  Otherwise, we call the triple $(a_i,a_j,a_k)$ \emph{good}.  Now, let us estimate the expected number of sets $A \subset \{0,1,\ldots, N - 1\}$ of size $n$ such that every triple in $A$ is good.   For a given triple $a_i,a_j,a_k \in \{0,1,\ldots, N - 1\}$, where $a_i < a_j < a_k$, the probability that $(a_i,a_j,a_k)$ is good is 7/8.  Let $A = \{a_1,\ldots, a_n\}$ be a set of $n$ vertices in $\{0,1,\ldots, N - 1\}$, where $a_1 < \cdots < a_n$.  Let $S$ be a partial Steiner $(n,3,2)$-system with vertex set $A$, that is, $S$ is a 3-uniform hypergraph such that each $2$-element set of vertices is contained in at most one edge in $S$.  Moreover, $S$ satisfies $|S| = c'n^{2}$.  It is known that such a system exists. Then the probability that every triple in $A$ is good is at most the probability that every edge in $S$ is good.  Since the edges in $S$ are independent, that is no two edges have more than one vertex in common, the probability that every triple in $A$ is good is at most $\left( \frac{7}{8}\right)^{|S|} \leq \left(\frac{7}{8}\right)^{c'n^{2}}$.   Therefore, the expected number of sets of size $n$ with every triple being good is at most

$${N\choose n}\left( \frac{7}{8}\right)^{c'n^{2}} < 1/3,$$

\noindent for an appropriate choice for $c$.  By Markov's inequality and the union bound, we can conclude that there is a coloring $\phi$ with the desired properties.  \end{proof}

Let $c> 0$ be the constant from the lemma above, and let $A = \{0,1,\ldots, \lfloor 2^{cn}\rfloor-1\}$ and $\phi:{A\choose 2}\rightarrow \{\textnormal{red, blue}\}$ be a 2-coloring of the pairs of $A$ with the properties described above.  Let $V = \{0,1,\ldots, N-1\}$, where $N  = 2^{\lfloor 2^{cn}\rfloor}$.  In what follows, we will use $\phi$ to define a red/blue coloring $\chi:{V\choose 4}\rightarrow \{\textnormal{red, blue}\}$ of the 4-tuples of $V$ such that $\chi$ does not produce a monochromatic red copy of $K^{(4)}_6$ and does not produce a monochromatic blue copy of $K^{(4)}_{32n^5}$.  This would imply the desired lower bound for $r_4(6,n)$.   For $v_1 < v_2 < v_3 < v_4$ and $\delta_i = \delta(v_i,v_{i + 1})$, we set $\chi(v_1,v_2,v_3,v_4) = $ red if

\begin{enumerate}

\item[(a)] $\delta_1,\delta_2,\delta_3$ forms a monotone sequence and the triple $(\delta_1,\delta_2,\delta_3)$ is \emph{bad}, that is, $\phi(\delta_1,\delta_2) = \phi(\delta_2,\delta_3) = $ blue and $\phi(\delta_1,\delta_3) = $ red, or

\item[(b)] $\delta_1 < \delta_2 > \delta_3$ and $\delta_1 = \delta_3$, or

\item[(c)]  $\delta_1 < \delta_2 > \delta_3$, $\delta_1 \neq \delta_3$, and the set $\{\delta_1,\delta_2,\delta_3\}$ is monochromatic with respect to $\phi$, or

\item[(d)] $\delta_1 > \delta_2 < \delta_3$, $\delta_1 < \delta_3$, and $\phi(\delta_3,\delta_1) = \phi(\delta_3,\delta_2)$ and $\phi(\delta_1,\delta_2) \neq \phi(\delta_3,\delta_1)$.

\end{enumerate}

 \noindent See Figure \ref{redex1} for small examples. Otherwise, $\chi(v_1,v_2,v_3,v_4)$ =  blue.

\def\r6nRED1{
\begin{tikzpicture}
  \node at (0,0)  {$v_4$:};
  \node at (0.5,0)  {};
  \node at (1,0)  {1};
  \node at (1.5,0)  {1};
  \node at (2,0)  {1};
  \node at (0,0.5) {$v_3$:};
  \node at (0.5,0.5)  {};
  \node at (1,0.5)  {1};
  \node at (1.5,0.5)  {1};
  \node at (2,0.5)  {0};
  \node at (0,1)  {$v_2$:};
  \node at (0.5,1)  {};
  \node at (1,1)  {1};
  \node at (1.5,1)  {0};
  \node at (2,1)  {0};
  \node at (0,1.5)  {$v_1$:};
  \node at (0.5,1.5)  {};
  \node at (1,1.5)  {0};
  \node at (1.5,1.5)  {0};
  \node at (2,1.5)  {0};
  \node at (0,2)  {};
  \node at (0.5,2) {};

  \node at (1,2) [vertex] {};
\node at (1,2.35)  {$\delta_1$};

\node at (1.5,2.35)  {$\delta_2$};

\node at (2,2.35)  {$\delta_3$};
  \node at (1.5,2) [vertex] {};
  \node at (2,2) [vertex] {};

    \draw[ultra thick, blue]  (1,2) to[out=0, in=180] (3/2,2);
    \draw[ultra thick, blue]  (3/2,2) to[out=0, in=180] (2,2);
    \draw[ultra thick, red]  (1,2) to[out=-30, in=-150] (2,2);

  \node at (3,0)  {1};
  \node at (3.5,0)  {1};
  \node at (4,0)  {1};

  \node at (3,0.5)  {0};
  \node at (3.5,0.5)  {1};
  \node at (4,0.5)  {1};

   \node at (3,1)  {0};
  \node at (3.5,1)  {0};
   \node at (4,1)  {1};

    \node at (3,1.5)  {0};
    \node at (3.5,1.5)  {0};
    \node at (4,1.5)  {0};

  \node at (3,2) [vertex] {};
\node at (3,2.35)  {$\delta_3$};

\node at (3.5,2.35)  {$\delta_2$};

\node at (4,2.35)  {$\delta_1$};
  \node at (3.5,2) [vertex] {};
  \node at (4,2) [vertex] {};

    \draw[ultra thick, blue]  (3,2) to[out=0, in=180] (7/2,2);
    \draw[ultra thick, blue]  (7/2,2) to[out=0, in=180] (4,2);
    \draw[ultra thick, red]  (3,2) to[out=-30, in=-150] (4,2);

     \node at (2.5, -1) {(a)};

  \node at (5.5,0)  {1};
  \node at (6,0)  {1};
  \node at (6.5,0)  {};

  \node at (5.5,0.5)  {1};
  \node at (6,0.5)  {0};

   \node at (5.5,1)  {0};
  \node at (6,1)  {1};
   \node at (6.5,1)  {};

    \node at (5.5,1.5)  {0};
    \node at (6,1.5)  {0};
    \node at (6.5,1.5)  {};

  \node at (5.5,2) [vertex] {};
\node at (5.25,2.35)  {$\delta_2$};

\node at (6.25,2.35)  {$\delta_1$=$\delta_3$};
  \node at (6,2) [vertex] {};
\draw[ultra thick, black]  (5.5,2) to[out=0, in=180] (6,2);

     \node at (5.75,-1)  {(b)};

  \node at (7.5,0)  {1};
  \node at (8,0)  {1};
  \node at (8.5,0)  {0};

  \node at (7.5,0.5)  {1};
  \node at (8,0.5)  {0};
  \node at (8.5,0.5)  {1};

   \node at (7.5,1)  {0};
  \node at (8,1)  {0};
   \node at (8.5,1)  {1};

    \node at (7.5,1.5)  {0};
    \node at (8,1.5)  {0};
    \node at (8.5,1.5)  {0};

  \node at (7.5,2) [vertex] {};
\node at (7.5,2.35)  {$\delta_2$};

\node at (8,2.35)  {$\delta_3$};

\node at (8.5,2.35)  {$\delta_1$};
  \node at (8,2) [vertex] {};
  \node at (8.5,2) [vertex] {};

    \draw[ultra thick, black]  (7.5,2) to[out=0, in=180] (8,2);
    \draw[ultra thick, black]  (8,2) to[out=0, in=180] (8.5,2);
    \draw[ultra thick, black]  (7.5,2) to[out=-30, in=-150] (8.5,2);

  \node at (9.5,0)  {1};
  \node at (10,0)  {1};
  \node at (10.5,0)  {1};

  \node at (9.5,0.5)  {1};
  \node at (10,0.5)  {1};
  \node at (10.5,0.5)  {0};

   \node at (9.5,1)  {0};
  \node at (10,1)  {1};
   \node at (10.5,1)  {1};

    \node at (9.5,1.5)  {0};
    \node at (10,1.5)  {0};
    \node at (10.5,1.5)  {1};

  \node at (9.5,2) [vertex] {};
\node at (9.5,2.35)  {$\delta_2$};

\node at (10,2.35)  {$\delta_1$};

\node at (10.5,2.35)  {$\delta_3$};
  \node at (10,2) [vertex] {};
  \node at (10.5,2) [vertex] {};

    \draw[ultra thick, black]  (9.5,2) to[out=0, in=180] (10,2);
    \draw[ultra thick, black]  (10,2) to[out=0, in=180] (10.5,2);
    \draw[ultra thick, black]  (9.5,2) to[out=-30, in=-150] (10.5,2);

     \node at (9,-1)  {(c)};

  \node at (12,0)  {1};
  \node at (12.5,0)  {0};
  \node at (13,0)  {0};

  \node at (12,0.5)  {0};
  \node at (12.5,0.5)  {1};
  \node at (13,0.5)  {1};

   \node at (12,1)  {0};
  \node at (12.5,1)  {1};
   \node at (13,1)  {0};

    \node at (12,1.5)  {0};
    \node at (12.5,1.5)  {0};
    \node at (13,1.5)  {1};

  \node at (12,2) [vertex] {};
\node at (12,2.35)  {$\delta_3$};

\node at (12.5,2.35)  {$\delta_1$};

\node at (13,2.35)  {$\delta_2$};
  \node at (12.5,2) [vertex] {};
  \node at (13,2) [vertex] {};

    \draw[ultra thick, blue]  (12,2) to[out=0, in=180] (12.5,2);
    \draw[ultra thick, red]  (12.5,2) to[out=0, in=180] (13,2);
    \draw[ultra thick, blue]  (12,2) to[out=-30, in=-150] (13,2);

  \node at (14,0)  {1};
  \node at (14.5,0)  {0};
  \node at (15,0)  {0};

  \node at (14,0.5)  {0};
  \node at (14.5,0.5)  {1};
  \node at (15,0.5)  {1};

   \node at (14,1)  {0};
  \node at (14.5,1)  {1};
   \node at (15,1)  {0};

    \node at (14,1.5)  {0};
    \node at (14.5,1.5)  {0};
    \node at (15,1.5)  {1};

  \node at (14,2) [vertex] {};
\node at (14,2.35)  {$\delta_3$};

\node at (14.5,2.35)  {$\delta_1$};

\node at (15,2.35)  {$\delta_2$};
  \node at (14.5,2) [vertex] {};
  \node at (15,2) [vertex] {};

    \draw[ultra thick, red]  (14,2) to[out=0, in=180] (14.5,2);
    \draw[ultra thick, blue]  (14.5,2) to[out=0, in=180] (15,2);
    \draw[ultra thick, red]  (14,2) to[out=-30, in=-150] (15,2);

     \node at (13.5,-1)  {(d)};

      \end{tikzpicture}}

\begin{figure}
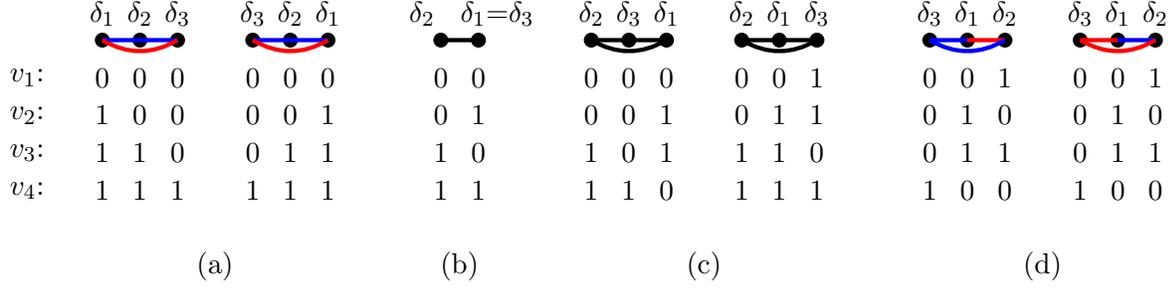

      {\r6nRED1} \caption{Examples of $v_1 < v_2 < v_3 < v_4$ and $\delta_1 = \delta(v_1,v_2), \delta_2 = \delta(v_2,v_3), \delta_3 = \delta(v_3,v_4)$, such that $\chi(v_1,v_2,v_3,v_4) = $ red.  For each case, $v_i$ is represented in binary form with the left-most entry being the most significant bit.}\label{redex1}
  \end{figure}

For the sake of contradiction, suppose that the coloring $\chi$ produces a red $K_6^{(4)}$ on vertices $v_1 < \cdots < v_6$, and let $\delta_i = \delta(v_i,v_{i + 1})$, $1 \leq i \leq 5$.  Let us first consider the following cases for $\delta_1,\ldots, \delta_4$, which corresponds to the vertices, $v_1,\ldots, v_5$.

\medskip

\emph{Case 1}.  Suppose that $\delta_1,\ldots, \delta_{4}$ forms a monotone sequence.  If $\delta_1 > \cdots > \delta_4$, then we have $\phi(\delta_1,\delta_3) = $ red since $\chi(v_1,v_2,v_3,v_4) = $ red.  However, this implies that $\chi(v_1,v_3,v_4,v_5) = $ blue since $\delta(v_1,v_3) = \delta_1$ by Property II, contradiction.  A similar argument follows if $\delta_1 < \cdots < \delta_4$.

\medskip

\emph{Case 2}.  Suppose $\delta_1 > \delta_2 > \delta_3 < \delta_4$.  By Property III, $\delta_4 \neq \delta_2,\delta_1$.  Since $\delta_1 > \delta_2 > \delta_3$, this implies that $\phi(\delta_1,\delta_2) = \phi(\delta_2,\delta_3) = $ blue and $\phi(\delta_1,\delta_3) = $ red.  Since $\delta(v_2,v_4) = \delta_2$ and $\chi(v_1,v_2,v_4,v_5) = $ red, we have $\delta_4 > \delta_1$.  Hence $\phi(\delta_4,\delta_3) = \phi(\delta_4,\delta_2) =$ red.  However, since $\delta(v_1,v_3) = \delta_1$ by Property II, we have $\chi(v_1,v_3,v_4,v_5)$ is blue, contradiction.

\medskip

  \emph{Case 3.} Suppose $\delta_1 < \delta_2 < \delta_3 > \delta_4$.   This implies that $\phi(\delta_1,\delta_2) = \phi(\delta_2,\delta_3) = $ blue and $\phi(\delta_1,\delta_3) = $ red.  Suppose $\delta_4 = \delta_2$.  Since $\delta(v_1,v_3)  = \delta_2$ and $\delta_2 < \delta_3 > \delta_4$, this implies the triple $(\delta_1,\delta_2,\delta_4)$ forms a monochromatic blue set with respect to $\phi$, which is a contradiction.  A similar argument follows in the case that $\delta_4 = \delta_1$.  So we can assume $\delta_4 \neq \delta_1,\delta_2$. Since $\chi(v_2,v_3,v_4,v_5) = $ red, the triple $\{\delta_2,\delta_3,\delta_4\}$ forms a monochromatic blue set with respect to $\phi$.  By Property II we have $\delta(v_2,v_4) = \delta_3$ and $\delta_1 < \delta(v_2,v_4) > \delta_4$.  This implies that $\chi(v_1,v_2,v_4,v_5) = $ blue, contradiction

  \medskip

\emph{Case 4}.  Suppose $\delta_1 < \delta_2 > \delta_3 > \delta_4$.   This implies that $\phi(\delta_2,\delta_3) = \phi(\delta_3,\delta_4) = $ blue and $\phi(\delta_2,\delta_4) = $ red.  Suppose $\delta_1 = \delta_3$.  By Property II, we have $\delta(v_2,v_4) = \delta_2$.  However, $\chi(v_1,v_2,v_4,v_5) = $ red and $\delta_1 < \delta_2 > \delta_4$ implies that the triple $(\delta_1,\delta_2,\delta_4)$ must form a monochromatic set with respect to $\phi$, contradiction.  A similar argument follows if $\delta_1 = \delta_4$.  Therefore, we can assume that $\delta_1 \neq \delta_3,\delta_4$.  Since $\chi(v_1,v_2,v_3,v_4) = $ red, the triple $(\delta_1,\delta_2,\delta_3)$ forms a monochromatic blue set with respect to $\phi$.   By Property II we have $\delta(v_2,v_4) = \delta_2$ and $\delta_1 < \delta(v_2,v_4) > \delta_4$.  This implies $\chi(v_1,v_2,v_4,v_5) = $ blue, contradiction.

\medskip

\emph{Case 5}.  Suppose $\delta_1 > \delta_2 < \delta_3 < \delta_4$.  Note that by Property III, $\delta_1 \neq \delta_3,\delta_4$.   Since $\delta_1,\delta_2,\delta_3$ forms a monotone sequence, this implies that $\phi(\delta_2,\delta_3) = \phi(\delta_3,\delta_4) = $ blue and $\phi(\delta_2,\delta_4) = $ red.  Moreover, we must have $\delta_1 < \delta_3$ since $\chi(v_1,v_2,v_3,v_4) = $ red.  Hence $\phi(\delta_3,\delta_1) = $ blue and $\phi(\delta_1,\delta_2) = $ red.  However, since $\delta(v_3,v_5) = \delta_4$, we have $\delta_1 > \delta_2 < \delta(v_3,v_5)$ and $\chi(v_1,v_2,v_3,v_5) = $ blue, contradiction.

\emph{Case 6}.  Suppose $\delta_1 < \delta_2 >  \delta_3 < \delta_4$.  Then we must also have $\delta_4 > \delta_2$ since $\chi(v_2,v_3,v_4,v_5) = $ red.  By Property II, $\delta(v_3,v_5) = \delta_4$ and we have $\delta_1 < \delta_2 < \delta(v_3,v_5)$.  Since $\chi(v_1,v_2,v_3,v_5) = $ red, we have $\phi(\delta_1,\delta_2) = \phi(\delta_2,\delta_4) = $ blue and $\phi(\delta_1,\delta_4) = $ red.  Since $\chi(v_1,v_2,v_3,v_4) = $ red, the triple $(\delta_1,\delta_2,\delta_3)$ forms a monochromatic blue set.  However, this implies that $\chi(v_2,v_3,v_4,v_5) = $ blue, contradiction.

\medskip

Now if $v_1,\ldots, v_5$ and $\delta_1,\ldots, \delta_4$ does not fall into one of the 6 cases above, then we must have $\delta_1 > \delta_2 < \delta_3 > \delta_4$.  However, this implies that $v_2,\ldots, v_6$ and $\delta_2,\ldots, \delta_5$ does fall into one of the 6 cases above, which implies our contradiction.  Therefore, $\chi$ does not produce a monochromatic red copy of $K^{(4)}_6$ in our 4-uniform hypergraph.

Next we show that there is no blue $K_m^{(4)}$ in coloring $\chi$, where $m = 32n^5$.  For the sake of contradiction, suppose we have vertices $v_1,\ldots, v_m \in V$ such that $v_1 < \cdots < v_m$, and $\chi$ colors every $4$-tuple in the set $\{v_1, \ldots, v_m\}$ blue.  Let $\delta_i = \delta(v_i,v_{i + 1})$ for $1\leq i \leq m - 1$.

Set $\delta^{\ast}_1 = \max\{\delta_1,\ldots, \delta_m\}$, where $\delta^{\ast}_1 = \delta(v_{i_1},v_{i_1 + 1})$.  Set $$V_1 = \{v_1,v_2,\ldots, v_{i_1}\} \hspace{.5cm}\textnormal{and}\hspace{.5cm} V_2 = \{v_{i_1 + 1},v_{i_1 + 2}, \ldots, v_{m}\}.$$  Now we establish the following lemma.

\begin{lemma}\label{ststart}

For any $W\subset \{v_1,\ldots, v_m\}$, where $W = W_1\cup W_2$ is a partition of $W$ described as above, either $|W_1| < m/2n$ or $|W_2| < m/2n$.  In particular,  either $|V_1| < m/2n$ or $|V_2| < m/2n$.

\end{lemma}

Before we prove Lemma \ref{ststart}, let us finish the argument that $\chi$ does not color every 4-tuple in the set $\{v_1,\ldots, v_m\}$ blue via the following lemma which will also be used later in the paper.

\begin{lemma}\label{ststart2}
If Lemma~\ref{ststart} holds, then $\chi$ colors a 4-tuple in the set $\{v_1,\ldots, v_m\}$ red.
\end{lemma}
\proof  We greedily construct a set $D_t = \{\delta^{\ast}_1, \delta^{\ast}_{2},\ldots, \delta^{\ast}_{t}\}\subset \{\delta_1, \delta_2, \ldots, \delta_m\}$ and a set $S_t \subset \{v_1,\ldots, v_m\}$ such that the following holds.

\begin{enumerate}
\item We have $\delta^{\ast}_{1} > \cdots > \delta^{\ast}_{t}$, where $\delta^{\ast}_j = \delta(v_{i_j},v_{i_j + 1})$.

\item The indices of the vertices in $S_t$ are consecutive, that is, $S_t = \{v_r,v_{r + 1},\ldots, v_{s-1},v_s\}$ for $1 \leq r < s \leq n$.  Moreover, $\delta^{\ast}_t > \max\{\delta_r,\delta_{r + 1},\ldots, \delta_{s-1}\}$.

\item $|S_t| > m - tm/2n$.

\item For each $\delta^{\ast}_{j}  = \delta(v_{i_j},v_{i_j + 1}) \in D_t$, consider the set of vertices $$S =  \{v_{i_{j + 1}},v_{i_{j + 1} + 1},v_{i_{j + 2}},v_{i_{j + 2} + 1}\ldots, v_{i_{t}}, v_{i_t + 1}\}\cup S_t.$$  Then either every element in $S$ is greater than $v_{i_j}$ or every element in $S$ is less than $v_{i_j + 1}$.  In the former case we will label $\delta^{\ast}_{j}$ \emph{white}, in the latter case we label it \emph{black}.

\end{enumerate}

We start with the $D_0 = \emptyset$ and $S_0 = \{v_1,\ldots, v_{m}\}$.  Having obtained $D_{t-1} = \{\delta^{\ast}_{1},\ldots, \delta^{\ast}_{t-1}\}$ and $S_{t-1} = \{v_r,\ldots,v_s\}$, where $1\leq r < s \leq n$, we construct $D_{t}$ and $S_{t}$ as follows.  Let $\delta^{\ast}_{t}  = \delta(v_{i_{t}},v_{i_{t} + 1})$ be the unique largest element in $\{\delta_r,\delta_{r + 1},\ldots, \delta_{s-1}\}$, and set $D_{t} = D_{t-1}\cup \delta^{\ast}_{t}$.   The uniqueness of $\delta^{\ast}_{t}$ follows from Properties I and II.  We partition $S_{t-1} = T_1\cup T_2$, where $T_1 = \{v_r, v_{r + 1}, \ldots, v_{i_{t}}\}$ and $T_2 = \{v_{i_{t } + 1}, v_{i_{t} + 2}, \ldots, v_s\}$.   By Lemma \ref{ststart}, either $|T_1| < m/2n$ or $|T_2| < m/2n$.  If $|T_1| < m/2n$, we set $S_t = T_2$ and label $\delta^{\ast}_t$ white.  Likewise, if $|T_2| < m/2n$, we set $S_t = T_1$ and label $\delta^{\ast}_t$ black.  By induction, we have $$|S_t| > |S_{t-1}| - m/2n \geq (m-(t-1)m/2n) - m/2n = m - tm/2n.$$

Since $|S_0| = m$ and $|S_{t}| \geq 1$ for $t = 2n$, we can construct $D_{2n} = \{\delta^{\ast}_{1},\ldots, \delta^{\ast}_{2n}\}$ with the desired properties.  By the pigeonhole principle, there are at least $n$ elements in $D_{2n}$ with the same label, say \emph{white}.  The other case will follow by a symmetric argument.  We remove all black labeled elements in $D_{2n}$, and let $\{\delta^{\ast}_{j_1},\ldots, \delta^{\ast}_{j_n}\}$ be the resulting set.

Now consider the vertices $v_{j_1}, v_{j_2}, \ldots, v_{j_{n}}, v_{j_n + 1} \in V$.  By construction and by Property II, we have $v_{j_1} < v_{j_2} < \cdots < v_{j_{n}} < v_{j_n + 1}$ and $\delta(v_{j_1},v_{j_2}) = \delta^{\ast}_{i_{j_1}}, \delta(v_{j_2},v_{j_3}) = \delta^{\ast}_{i_{j_2}}, \ldots, \delta(v_{j_{n }}, v_{j_{n +1}}) = \delta^{\ast}_{i_{j_{n}}} $.  Therefore, we have a monotone sequence $$\delta(v_{j_1},v_{j_2}) > \delta(v_{j_2},v_{j_3}) > \cdots >  \delta(v_{j_{n }}, v_{j_{n +1}}).$$

By Lemma \ref{start}, there is a bad triple in the set $\{\delta^{\ast}_{j_1},\ldots, \delta^{\ast}_{j_n}\}$ with respect to $\phi$.  By Property IV, $\chi$ does not color every 4-tuple in $V = \{v_1,\ldots, v_m\}$ blue, which completes the proof of Lemma \ref{ststart2}. \qed

\medskip

Now let us go back and prove Lemma \ref{ststart}.  First, we make the following observation.

\begin{observation}\label{distinct} Let $v_1 < \cdots < v_m \in V$ such that $\chi$ colors every 4-tuple in the set $\{v_1,\ldots, v_m\}$ blue.  Then for $\delta_i = \delta(v_i,v_{i + 1})$, $\delta_i \neq \delta_j$ for $1 \leq i < j < m$.

\end{observation}

\begin{proof}

For the sake of contradiction, suppose $\delta_i   = \delta_j$ for $i\neq j$.    By Property I, $j \neq i+1$.  Without loss of generality, we can assume that for all $r$ such that $i < r < j$, $\delta_r \neq \delta_i$.  Set $\delta_r = \max\{\delta_{i+1},\delta_{i+2},\ldots, \delta_{j-1}\}$, and notice that $\delta(v_{i+1},v_j) = \delta_r$ by Property II.  Now if $\delta_r > \delta_i  = \delta_j$, then $\chi(v_i,v_{i+1},v_j,v_{j+1}) = $ red and we have a contradiction. If $\delta_r < \delta_i$, then this would contradict Property III. Hence, the statement follows.\end{proof}

\medskip

\noindent \emph{Proof of Lemma \ref{ststart}.}  It suffices to show that the statement holds when $W_1  = V_1$ and $W_2 = V_2$.  For the sake of contradiction, suppose $|V_1|,|V_2| \geq m/2n = 16n^4$.  Recall that  $\delta^{\ast}_1 = \delta(v_{i_1},v_{i_1 + 1})$, $V_1 = \{v_1,v_2,\ldots, v_{i_1}\}$, $V_2 = \{v_{i_1 + 1},v_{i_1 + 1}, \ldots, v_{m}\}$, and set $A_1 = \{\delta_1,\ldots, \delta_{i_1 - 1}\}$ and $A_2 = \{\delta_{i_1 + 1},\ldots, \delta_{m-1}\}$.  For $i \in \{1,2\}$, let us partition $A_i = A_i^r\cup A_i^b$ where
$$A_i^r = \{\delta_j \in A_i: \phi(\delta^{\ast}_1,\delta_j) = \textnormal{ red}\}\qquad \hbox{ and } \qquad A^b_i = \{\delta_j \in A_i: \phi(\delta^{\ast}_1,\delta_j) = \textnormal{ blue}\}.$$  By the pigeonhole principle, either $|A^b_2| \geq 8n^4$ or $|A^r_2| \geq 8n^4$.  Without loss of generality, we can assume that $|A^b_2| \geq 8n^4$  since a symmetric argument would follow otherwise.

Fix $\delta_{j_1} \in A_1^b$ and $\delta_{j_2} \in A_2^b$, and recall that $\delta_{j_1} = \delta(v_{j_1},v_{j_1 + 1})$ and $\delta_{j_2} = \delta(v_{j_2},v_{j_2 + 1})$.  By Observation~\ref{distinct},  $\delta_{j_1} \neq \delta_{j_2}$, and by Property II, we have $\delta(v_{j_1+1},v_{j_2}) = \delta^{\ast}_{1}$.  Since $\chi(v_{j_1},v_{j_1+1},v_{j_2},v_{j_2+1}) = $ blue, this implies that $\phi(\delta_{j_1},\delta_{j_2}) = $ red.   By Lemma \ref{start} and Observation \ref{distinct}, we have $|A_1^b| < n$.  Indeed, otherwise we would have a monochromatic red copy of $K_{n,n}$ in $A$ with respect to $\phi$.  Therefore we have $|A_1^r| \geq 16n^4 - n - 1$.  Again by the pigeonhole principle, there is a subset $B \subset A_1^r$ of size at least $(16n^4 -n - 1)/n \geq 8n^3 - 1$, such that $B = \{\delta_j, \delta_{j + 1}, \ldots, \delta_{j + 8n^3 - 2}\}$, and whose corresponding vertices are $U= \{v_j,v_{j+1},\ldots, v_{j + 8n^3-1}\}$.  For simplicity and without loss of generality, let us rename $U= \{u_1,\ldots, u_{8n^3}\}$ and $\delta_i = \delta(u_i,u_{i+1})$ for $1 \leq i \leq 8n^3 - 1$.

Just as before, we greedily construct a set $D_t = \{\delta^{\ast}_{1},\ldots, \delta^{\ast}_{t}\}\subset \delta^{\ast}_1\cup \{\delta_1,\ldots, \delta_{8n^3 - 1}\}$ and a set $S_t \subset \{u_1,\ldots, u_{8n^3}\}$ such that the following holds.

\begin{enumerate}
\item We have $\delta^{\ast}_{1} > \cdots > \delta^{\ast}_{t}$, where $\delta^{\ast}_j = \delta(u_{i_j},u_{i_j + 1})$ for $i\geq 2$.

\item For each $\delta^{\ast}_{j}  = \delta(u_{i_j},u_{i_j + 1}) \in D_t$, consider the set of vertices $$S =  \{u_{i_{j + 1}},u_{i_{j + 1} + 1},\ldots, u_{i_{h}}, u_{i_h + 1}\}\cup S_t.$$  Then either every element in $S$ is greater than $u_{i_j}$ or every element in $S$ is less than $u_{i_j + 1}$.  In the former case we will label $\delta^{\ast}_{i_j}$ \emph{white}, in the latter case we label it \emph{black}.

\item The indices of the vertices in $S_t$ are consecutive, that is, $S_t = \{u_r,u_{r + 1},\ldots, u_{s-1},u_s\}$ for $1 \leq r < s \leq n$.  Set $B_t = \{\delta_r,\delta_{r+1},\ldots, \delta_{u_{s - 1}}\}$.

\item for each $\delta^{\ast}_{j} \in D_t$, either $\phi(\delta^{\ast}_{j},\delta) = $ red for every $\delta \in \{\delta^{\ast}_{j + 1},\delta^{\ast}_{j + 2},\ldots, \delta^{\ast}_t\}\cup B_t$, or $\phi(\delta^{\ast}_{j},\delta) = $ blue for every $\delta \in \{\delta^{\ast}_{j + 1},\delta^{\ast}_{j + 2},\ldots, \delta^{\ast}_t\}\cup B_t$.

\item We have $|S_t| \geq 8n^3  - (t-1)2n^2$.

\end{enumerate}

We start with $S_1 = U = \{u_1,\ldots, u_{8n^3}\}$ and $D_1 = \{\delta^{\ast}_1\}$, where we recall that $\delta^{\ast}_1 =\delta(v_{i_1},v_{i_1 + 1})$.  Having obtained $D_{t-1} = \{\delta^{\ast}_{1},\ldots, \delta^{\ast}_{t-1}\}$ and $S_{t-1} = \{u_r,\ldots,u_s\}$, $1\leq r < s \leq n$, we construct $D_{t}$ and $S_{t}$ as follows.  Let $\delta^{\ast}_{t}  = \delta(u_{i_{t}},u_{i_{t} + 1})$ be the unique largest element in $\{\delta_r,\delta_{r + 1},\ldots, \delta_{s-1}\}$, and set $D_{t} = D_{t-1}\cup \delta^{\ast}_{t}$.  The uniqueness of $\delta^{\ast}_{t}$ follows from Properties I and II. Let us partition $S_{t} = T_1\cup T_2$, where $T_1 = \{u_r, u_{r + 1}, \ldots, u_{i_{t}}\}$ and $T_2 = \{u_{i_{t} + 1}, u_{i_{h + 1} + 2}, \ldots, u_s\}$.   Now we make the following observation.

\begin{observation}\label{ul}

We have $|T_1| < 2n^2$ or $|T_2| < 2n^2$.

\end{observation}

\begin{proof} For the sake of contradiction, suppose $|T_1|,|T_2| \geq 2n^2$ and let $B_1 = \{\delta_r,\delta_{r + 1}, \ldots, \delta_{i_t - 1}\}$ and $B_2 = \{\delta_{i_1 + 1},\delta_{i_t + 2}, \ldots, \delta_{s - 1}\}$.   Notice that for every $\delta \in B_2$ we have $\phi(\delta^{\ast}_{t}, \delta) = $ red.  Indeed, suppose for $\delta = \delta(u_{\ell},u_{\ell + 1}) \in B_2$ we have $\phi(\delta^{\ast}_{t}, \delta) = $ blue.  Recall $\delta^{\ast}_1 = \delta(v_{i_1},v_{i_1 + 1})$, $\delta^{\ast}_t = \delta(u_{i_t},u_{i_t + 1})$, where $$u_{i_t} < u_{i_t + 1} <  u_{\ell} < u_{\ell + 1} < v_{i_1} < v_{i_1 + 1}.$$  Consider the vertices $ v_{i_1 + 1},u_{i_{t}},u_{\ell},u_{\ell + 1}$.  By definition of $\chi$, we have $\chi(u_{i_t},u_{\ell},u_{\ell + 1},v_{i_1 + 1})  =  $ red, contradiction.  Therefore, by the same argument as above, there are less than $n$ elements $\delta \in B_1$ such that $\phi(\delta^{\ast}_{t},\delta)  = $ red.  Since $|T_1| > 2n^2$, by the pigeonhole principle, there is a set of $n + 1$ consecutive vertices $\{u_{\ell}, u_{\ell + 1}, \ldots, u_{\ell + n}\} \subset T_1$ and the subset $\{\delta_{\ell},\delta_{\ell + 1},\ldots, \delta_{\ell + n -1}\}\subset B_1$ such that $\phi(\delta^{\ast}_{t}, \delta) = $ blue for every $\delta \in \{\delta_{\ell},\delta_{\ell + 1},\ldots, \delta_{\ell + n - 1}\}$.  Notice that

$$\delta_{\ell} < \delta_{\ell + 1} < \cdots < \delta_{\ell + n-1}.$$

Indeed, suppose $\delta_r > \delta_{r+1}$ for some $r \in \{\ell, \ell + 1,\ldots, \ell+ n - 2\}$.  Then $\phi(\delta_r, \delta_{r + 1}) = $ red implies that $\chi(u_{i_t + 1}, u_r,u_{r + 1},u_{r + 2}) = $ red, contradiction.  Likewise if $\phi(\delta_r, \delta_{r + 1}) = $ blue, then $\chi(v_{i_1 + 1}, u_r,u_{r + 1},u_{r + 2}) = $ red, contradiction.  However, by Lemma \ref{start}, there is a bad triple in $\{\delta_{\ell}, \delta_{\ell + 1}, \ldots, \delta_{\ell + n - 1}\}$ with respect to $\phi$.  Since $\delta_{\ell}, \delta_{\ell + 1}, \ldots, \delta_{\ell + n - 1}$ forms a monotone sequence, by Property IV, $\chi$ colors some 4-tuple in the set $\{u_{\ell},u_{\ell + 1},\ldots, u_{\ell + n}\}$ red, contradiction.  Hence the statement follows. \end{proof}

If $|T_1| < 2n^2$, we set $S_t = T_2$.  Otherwise by Observation \ref{ul} we have $|T_2| < 2n^2$ and we set $S_t = T_1$.  Hence $|S_t| > |S_{t-1}| - 2n^2$.

Since $|S_1| = |U| =8n^3$, we have $|S_t| > 0$ for $t = 2n$.  Therefore, we can construct $D_{2n} = \{\delta^{\ast}_{1},\ldots, \delta^{\ast}_{2n}\}$ with the desired properties.  By the pigeonhole principle, at least $n$ elements in $D_{2n}$ have the same label, say \emph{white}.  The other case will follow by a symmetric argument.   We remove all black labeled elements in $D_{2n}$, and let $\{\delta^{\ast}_{j_1},\ldots, \delta^{\ast}_{j_n}\}$ be the resulting set, and for simplicity, let $\delta^{\ast}_{j_r} = \delta(v_{j_r},v_{j_r + 1})$.

Now consider the vertices $v_{j_1}, v_{j_2}, \ldots, v_{j_{n}}, v_{j_n + 1} \in V$.  By construction and by Property II, we have $v_{j_1} < v_{j_2} < \cdots < v_{j_{n}} < v_{j_n + 1}$ and $\delta(v_{j_1},v_{j_2}) = \delta^{\ast}_{i_{j_1}}, \delta(v_{j_2},v_{j_3}) = \delta^{\ast}_{i_{j_2}}, \ldots, \delta(v_{j_{n }}, v_{j_{n +1}}) = \delta^{\ast}_{i_{j_{n}}} $.  Therefore, we have a monotone sequence $$\delta(v_{j_1},v_{j_2}) > \delta(v_{j_2},v_{j_3}) > \cdots >  \delta(v_{j_{n }}, v_{j_{n +1}}).$$

By Lemma \ref{start}, there is a bad triple in the set $\{\delta^{\ast}_{j_1},\ldots, \delta^{\ast}_{j_n}\}$ with respect to $\phi$.  By Property IV, $\chi$ does not color every 4-tuple in $V = \{v_1,\ldots, v_m\}$ blue which is a contradiction.  \qed

\section{A new lower bound for $r_4(5,n)$}

Again we apply a variant to the Erd\H os-Hajnal stepping up lemma in order to establish a new lower bound for $r_4(5,n)$.  We will use the following lemma.

 \begin{lemma}\label{offdiag}
 For $n\geq 5$, there is an absolute constant $c> 0$ such that the following holds.  For $N = \lfloor n^{c\log n} \rfloor $, there is a red/blue coloring $\phi$ on the pairs of $\{0,1,\ldots, N - 1\}$ such that

 \begin{enumerate}

 \item there is no monochromatic red copy of $K_{\lfloor \log n \rfloor}$,

 \item there are no two disjoint $n$-sets $A,B \subset \{0,1,\ldots, N - 1\}$, such that $\phi(a,b) = $ blue for every $a\in A$ and $b\in B$ (i.e. no blue $K_{n,n}$).

\item there is no $n$-set $A\subset \{0,1,\ldots,  N - 1\}$ such that every triple $a_i,a_j,a_k \in A$, where $a_i < a_j < a_k$, avoids the pattern $\phi(a_i,a_j) =\phi (a_j,a_k) = blue$ and $\phi(a_i,a_k) = red$.

 \end{enumerate}

 \end{lemma}

We omit the proof of Lemma \ref{offdiag}, which follows by the same probabilistic argument used for Lemma~\ref{start}.  For the reader's convenience, let us restate the result that we are about to prove.

\begin{theorem}
For $n \geq 5$, there is an absolute constant $c>0$ such that $r_4(5,n) > 2^{n^{c\log n}}$.

\end{theorem}

\proof Let $c> 0$ be the constant from Lemma \ref{offdiag}, and set $A = \{0,1,\ldots, \lfloor n^{c\log n}\rfloor -1\}$.  Let $\phi$ be the red/blue coloring on the pairs of $A$ with the properties described in Lemma \ref{offdiag}.  Set $N = 2^{\lfloor n^{c\log n}\rfloor}$ and let $V = \{0,1,\ldots, N-1\}$.   In what follows, we will use $\phi$ to define a red/blue coloring $\chi:{V\choose 4}\rightarrow \{\textnormal{red, blue}\}$ of the 4-tuples of $V$ such that $\chi$ does not produce a monochromatic red $K_5^{(4)}$, and does not produce a monochromatic blue copy of $K_{2n^4}^{(4)}$.  This would imply the desired lower bound for $r_4(5,n)$.

Just as in the previous section, for any $v \in V$, we write $v=\sum_{i=0}^{\lfloor n^{c\log n}\rfloor -1}v(i)2^i$ with $v(i) \in \{0,1\}$ for each $i$.  For $u \not = v$, set $\delta(u,v) \in A$ denote the largest $i$ for which $u(i) \not = v(i)$.   Let $v_1,v_2,v_3,v_4 \in V$ such that $v_1 < v_2< v_3< v_4$ and set $\delta_i = \delta(v_i,v_{i+1})$.  We define $\chi(v_1,v_2,v_3,v_4)$ =  red if

\begin{enumerate}

\item[(a)] $\delta_1,\delta_2,\delta_3$ is monotone and $\phi(\delta_1,\delta_2) = \phi(\delta_2,\delta_3) = $ blue and $\phi(\delta_1,\delta_3) = $ red, or

\item[(b)] $\delta_1 < \delta_2 > \delta_3$ and $\delta_1 = \delta_3$, or

\item[(c)]  $\delta_1 < \delta_2 > \delta_3$, $\delta_1 \neq \delta_3$, and $\phi(\delta_1,\delta_2)  = \phi(\delta_2,\delta_3) = $ blue and $\phi(\delta_1,\delta_3) = $ red, or

\item[(d)] $\delta_1 > \delta_2 < \delta_3$, $\delta_1 < \delta_3$, and $\phi(\delta_1,\delta_3) = \phi(\delta_2,\delta_3) = $ red and $\phi(\delta_1,\delta_2) = $ blue, or

\item[(e)] $\delta_1 > \delta_2 < \delta_3$, $\delta_1 > \delta_3$, and $\phi(\delta_1,\delta_3) = \phi(\delta_1,\delta_2) = $ red and $\phi(\delta_2,\delta_3) = $ blue.

\def\r5nRED1{
\begin{tikzpicture}
  \node at (0,0)  {$v_4$:};
  \node at (0.5,0)  {};
  \node at (1,0)  {1};
  \node at (1.5,0)  {1};
  \node at (2,0)  {1};
  \node at (0,0.5) {$v_3$:};
  \node at (0.5,0.5)  {};
  \node at (1,0.5)  {1};
  \node at (1.5,0.5)  {1};
  \node at (2,0.5)  {0};
  \node at (0,1)  {$v_2$:};
  \node at (0.5,1)  {};
  \node at (1,1)  {1};
  \node at (1.5,1)  {0};
  \node at (2,1)  {0};
  \node at (0,1.5)  {$v_1$:};
  \node at (0.5,1.5)  {};
  \node at (1,1.5)  {0};
  \node at (1.5,1.5)  {0};
  \node at (2,1.5)  {0};
  \node at (0,2)  {};
  \node at (0.5,2) {};

  \node at (1,2) [vertex] {};
\node at (1,2.35)  {$\delta_1$};

\node at (1.5,2.35)  {$\delta_2$};

\node at (2,2.35)  {$\delta_3$};
  \node at (1.5,2) [vertex] {};
  \node at (2,2) [vertex] {};

    \draw[ultra thick, blue]  (1,2) to[out=0, in=180] (3/2,2);
    \draw[ultra thick, blue]  (3/2,2) to[out=0, in=180] (2,2);
    \draw[ultra thick, red]  (1,2) to[out=-30, in=-150] (2,2);

  \node at (3,0)  {1};
  \node at (3.5,0)  {0};
  \node at (4,0)  {0};

  \node at (3,0.5)  {0};
  \node at (3.5,0.5)  {1};
  \node at (4,0.5)  {0};

   \node at (3,1)  {0};
  \node at (3.5,1)  {0};
   \node at (4,1)  {1};

    \node at (3,1.5)  {0};
    \node at (3.5,1.5)  {0};
    \node at (4,1.5)  {0};

  \node at (3,2) [vertex] {};
\node at (3,2.35)  {$\delta_3$};

\node at (3.5,2.35)  {$\delta_2$};

\node at (4,2.35)  {$\delta_1$};
  \node at (3.5,2) [vertex] {};
  \node at (4,2) [vertex] {};

    \draw[ultra thick, blue]  (3,2) to[out=0, in=180] (7/2,2);
    \draw[ultra thick, blue]  (7/2,2) to[out=0, in=180] (4,2);
    \draw[ultra thick, red]  (3,2) to[out=-30, in=-150] (4,2);

     \node at (2.5, -1) {(a)};

  \node at (5.5,0)  {1};
  \node at (6,0)  {1};
  \node at (6.5,0)  {};

  \node at (5.5,0.5)  {1};
  \node at (6,0.5)  {0};

   \node at (5.5,1)  {0};
  \node at (6,1)  {1};
   \node at (6.5,1)  {};

    \node at (5.5,1.5)  {0};
    \node at (6,1.5)  {0};
    \node at (6.5,1.5)  {};

  \node at (5.5,2) [vertex] {};
\node at (5.25,2.35)  {$\delta_2$};

\node at (6.25,2.35)  {$\delta_1$=$\delta_3$};
  \node at (6,2) [vertex] {};
\draw[ultra thick, black]  (5.5,2) to[out=0, in=180] (6,2);

     \node at (5.75,-1)  {(b)};

  \node at (7.5,0)  {1};
  \node at (8,0)  {1};
  \node at (8.5,0)  {0};

  \node at (7.5,0.5)  {1};
  \node at (8,0.5)  {0};
  \node at (8.5,0.5)  {0};

   \node at (7.5,1)  {0};
  \node at (8,1)  {1};
   \node at (8.5,1)  {1};

    \node at (7.5,1.5)  {0};
    \node at (8,1.5)  {1};
    \node at (8.5,1.5)  {0};

  \node at (7.5,2) [vertex] {};
\node at (7.5,2.35)  {$\delta_2$};

\node at (8,2.35)  {$\delta_3$};

\node at (8.5,2.35)  {$\delta_1$};
  \node at (8,2) [vertex] {};
  \node at (8.5,2) [vertex] {};

    \draw[ultra thick, blue]  (7.5,2) to[out=0, in=180] (8,2);
    \draw[ultra thick, red]  (8,2) to[out=0, in=180] (8.5,2);
    \draw[ultra thick, blue]  (7.5,2) to[out=-30, in=-150] (8.5,2);

  \node at (9.5,0)  {1};
  \node at (10,0)  {0};
  \node at (10.5,0)  {1};

  \node at (9.5,0.5)  {1};
  \node at (10,0.5)  {0};
  \node at (10.5,0.5)  {0};

   \node at (9.5,1)  {0};
  \node at (10,1)  {1};
   \node at (10.5,1)  {0};

    \node at (9.5,1.5)  {0};
    \node at (10,1.5)  {0};
    \node at (10.5,1.5)  {0};

  \node at (9.5,2) [vertex] {};
\node at (9.5,2.35)  {$\delta_2$};

\node at (10,2.35)  {$\delta_1$};

\node at (10.5,2.35)  {$\delta_3$};
  \node at (10,2) [vertex] {};
  \node at (10.5,2) [vertex] {};

    \draw[ultra thick, blue]  (9.5,2) to[out=0, in=180] (10,2);
    \draw[ultra thick, red]  (10,2) to[out=0, in=180] (10.5,2);
    \draw[ultra thick, blue]  (9.5,2) to[out=-30, in=-150] (10.5,2);

     \node at (9,-1)  {(c)};

  \node at (12,0)  {1};
  \node at (12.5,0)  {0};
  \node at (13,0)  {0};

  \node at (12,0.5)  {0};
  \node at (12.5,0.5)  {1};
  \node at (13,0.5)  {1};

   \node at (12,1)  {0};
  \node at (12.5,1)  {1};
   \node at (13,1)  {0};

    \node at (12,1.5)  {0};
    \node at (12.5,1.5)  {0};
    \node at (13,1.5)  {1};

  \node at (12,2) [vertex] {};
\node at (12,2.35)  {$\delta_3$};

\node at (12.5,2.35)  {$\delta_1$};

\node at (13,2.35)  {$\delta_2$};
  \node at (12.5,2) [vertex] {};
  \node at (13,2) [vertex] {};

    \draw[ultra thick, red]  (12,2) to[out=0, in=180] (12.5,2);
    \draw[ultra thick, blue]  (12.5,2) to[out=0, in=180] (13,2);
    \draw[ultra thick, red]  (12,2) to[out=-30, in=-150] (13,2);

     \node at (12.5,-1)  {(d)};

  \node at (14.5,0)  {1};
  \node at (15,0)  {1};
  \node at (15.5,0)  {0};

  \node at (14.5,0.5)  {1};
  \node at (15,0.5)  {0};
  \node at (15.5,0.5)  {1};

   \node at (14.5,1)  {1};
  \node at (15,1)  {0};
   \node at (15.5,1)  {0};

    \node at (14.5,1.5)  {0};
    \node at (15,1.5)  {0};
    \node at (15.5,1.5)  {1};

  \node at (14.5,2) [vertex] {};
\node at (14.5,2.35)  {$\delta_1$};

\node at (15,2.35)  {$\delta_3$};

\node at (15.5,2.35)  {$\delta_2$};
  \node at (15,2) [vertex] {};
  \node at (15.5,2) [vertex] {};

    \draw[ultra thick, red]  (14.5,2) to[out=0, in=180] (15,2);
    \draw[ultra thick, blue]  (15,2) to[out=0, in=180] (15.5,2);
    \draw[ultra thick, red]  (14.5,2) to[out=-30, in=-150] (15.5,2);

     \node at (15,-1)  {(e)};

      \end{tikzpicture}}
\medskip

\begin{figure}
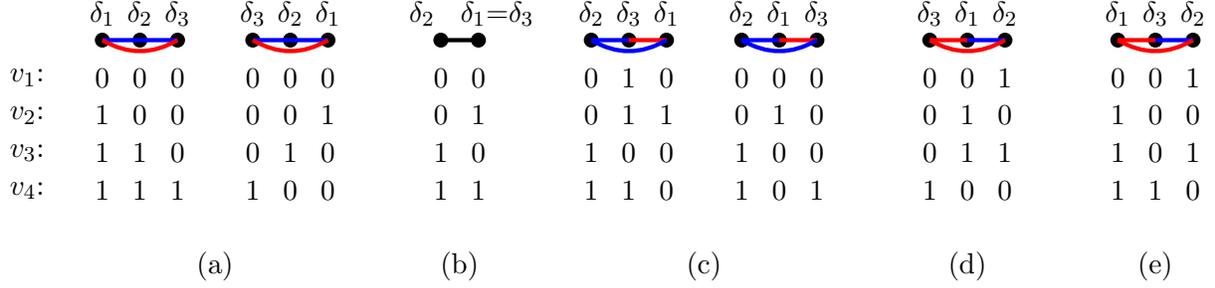

      {\r5nRED1}\caption{Examples of $v_1 < v_2 < v_3 < v_4$ and $\delta_1 = \delta(v_1,v_2), \delta_2 = \delta(v_2,v_3), \delta_3 = \delta(v_3,v_4)$, such that $\chi(v_1,v_2,v_3,v_4) = $ red.  For each case, $v_i$ is represented in binary form with the left-most entry being the most significant bit.}\label{redex2}
  \end{figure}

\end{enumerate}

\noindent See Figure \ref{redex2} for small examples.  Otherwise, $\chi(v_1,v_2,v_3,v_4)$ =  blue.

For the sake of contradiction, suppose that the coloring $\chi$ produces a red $K_5^{(4)}$ on vertices $v_1 < \cdots < v_5$, and let $\delta_i = \delta(v_i,v_{i + 1})$, $1 \leq i \leq 4$.  The proof now falls into the following cases, similar to the previous section.

\medskip

\emph{Case 1}.  Suppose that $\delta_1,\ldots, \delta_{4}$ forms a monotone sequence.  If $\delta_1 > \cdots > \delta_4$, then we have $\phi(\delta_1,\delta_3) = $ red since $\chi(v_1,v_2,v_3,v_4) = $ red.  However, this implies that $\chi(v_1,v_3,v_4,v_5) = $ blue since $\delta(v_1,v_3) = \delta_1$ by Property II, contradiction.  A similar argument follows if $\delta_1 < \cdots < \delta_4$.

\medskip

\emph{Case 2}.  Suppose $\delta_1 > \delta_2 > \delta_3 < \delta_4$.  By Property III, $\delta_4 \neq \delta_2,\delta_1$.  Since $\delta_1 > \delta_2 > \delta_3$, this implies that $\phi(\delta_1,\delta_2) = \phi(\delta_2,\delta_3) = $ blue and $\phi(\delta_1,\delta_3) = $ red.  Now consider the following subcases for $\delta_4$.

\medskip

\emph{Case 2.a}.  Suppose $\delta_4 > \delta_1$.  By Property II, $\delta(v_2,v_4) = \delta_2$.  Since $\chi(v_1,v_2,v_4,v_5) = $ red, this implies that $\phi(\delta_4,\delta_1) = \phi(\delta_4,\delta_2) = $ red.  However, since $\delta_1 = \delta(v_1,v_3)$, this implies $\chi(v_1,v_3,v_4,v_5) = $ blue, contradiction.

\medskip

\emph{Case 2.b}.  Suppose $\delta_2 < \delta_4 < \delta_1$.  Since $\chi(v_2,v_3,v_4,v_5) = $ red, we have $\phi(\delta_4,\delta_2) = \phi(\delta_4,\delta_3) = $ red.  However, this implies that $\chi(v_1,v_2,v_4,v_5) = $ blue since $\delta(v_2,v_4) = \delta_2$, contradiction.

\medskip

\emph{Case 2.c}. Suppose $\delta_3 < \delta_4 < \delta_2$.  Then this would imply $\chi(v_2,v_3,v_4,v_5) = $ blue, contradiction.

  \emph{Case 3.} Suppose $\delta_1 < \delta_2 < \delta_3 > \delta_4$.   This implies that $\phi(\delta_1,\delta_2) = \phi(\delta_2,\delta_3) = $ blue and $\phi(\delta_1,\delta_3) = $ red.  Hence we have $\delta_4 \neq \delta_1,\delta_2$.  Since $\delta(v_2,v_4) = \delta_3$, we have $\chi(v_1,v_2,v_4,v_5) = $ blue, contradiction.

\emph{Case 4}.  Suppose $\delta_1 < \delta_2 > \delta_3 > \delta_4$.   This implies that $\phi(\delta_2,\delta_3) = \phi(\delta_3,\delta_4) = $ blue and $\phi(\delta_2,\delta_4) = $ red.    Hence we have $\delta_1 \neq \delta_3,\delta_4$.  Since $\delta(v_2,v_4) = \delta_2$, we have $\chi(v_1,v_2,v_4,v_5) = $ blue, contradiction.

\medskip

\emph{Case 5}.  Suppose $\delta_1 > \delta_2 < \delta_3 < \delta_4$.  Note that by Property III, $\delta_1 \neq \delta_3,\delta_4$.   Since $\delta_2,\delta_3,\delta_4$ forms a monotone sequence, this implies that $\phi(\delta_2,\delta_3) = \phi(\delta_3,\delta_4) = $ blue and $\phi(\delta_2,\delta_4) = $ red.  Now we consider the following subcases for $\delta_1$.

\medskip

\emph{Case 5.a}.  Suppose $\delta_2 < \delta_1 < \delta_3$.  Then we have $\chi(v_1,v_2,v_3,v_4) = $ blue which is a contradiction.

\medskip

\emph{Case 5.b}. Suppose $\delta_3 < \delta_1 < \delta_4$.  Then we have $\phi(\delta_1,\delta_3) = \phi(\delta_1,\delta_2) = $ red.  Notice that $\delta(v_2,v_4) = \delta_3$ by Property II.  Therefore $\chi(v_1,v_2,v_4,v_5) = $ blue, contradiction.

\medskip

\emph{Case 5.c}. Suppose $\delta_1 > \delta_ 4$.   Then we have $\phi(\delta_1,\delta_3) = \phi(\delta_1,\delta_2) = $ red.  By Property II, $\delta(v_3,v_5) = \delta_4$ which implies $\chi(v_1,v_2,v_3,v_5) = $ blue, contradiction.

\emph{Case 6}.  Suppose $\delta_1 < \delta_2 >  \delta_3 < \delta_4$.   Then $\chi(v_1,v_2,v_3,v_4)= $ red implies that $\phi(\delta_2,\delta_1) = \phi(\delta_2,\delta_3) = $ blue and $\phi(\delta_1,\delta_3) = $ red.   Now if $\delta_2 < \delta_4$, $\chi(v_2,v_3,v_4,v_5) = $ red implies that $\phi(\delta_4,\delta_2) = \phi(\delta_4,\delta_3) = $ red.  By Property II, we have $\delta(v_2,v_4) = \delta_2$, and therefore $\delta_1 < \delta_2 < \delta_4$.  However, this implies $\chi(v_1,v_2,v_4,v_5) = $ blue, contradiction.  Now if $\delta_4 < \delta_2$, then $\chi(v_2,v_3,v_4,v_5) = $ blue, which is again a contradiction.

\medskip

\emph{Case 7}.  Suppose $\delta_1 > \delta_2 < \delta_3 > \delta_4$.  Then $\chi(v_2,v_3,v_4,v_5)= $ red implies that $\phi(\delta_3,\delta_2) = \phi(\delta_3,\delta_4) = $ blue and $\phi(\delta_2,\delta_4) = $ red.  Now if $\delta_1 < \delta_3$, then $\chi(v_1,v_2,v_3,v_4) = $ blue which is a contradiction.  Therefore we can assume that $\delta_1 > \delta_3$.  Since $\chi(v_1,v_2,v_3,v_4) = $ red we have $\phi(\delta_1,\delta_3) = $ red.  By Property II, $\delta(v_1,v_3) = \delta_1$ and $\delta_1 > \delta_3 > \delta_4$.  This implies that $\chi(v_1,v_3,v_4,v_5) = $ blue which is a contradiction.

Next we show that there is no blue $K_m^{(4)}$ in coloring $\chi$, where $m =  2n^4$.   We will prove this statement via the following claims.

\begin{claim}\label{c1} There do not exist vertices $w_1<\cdots< w_n$ in $V$ such that
$\phi(\delta(w_i,w_j), \delta(w_j, w_k))=$ red for every $i<j<k$.
\end{claim}

\begin{proof} Suppose for contradiction that these vertices $w_1 < \cdots < w_n$ exist.  Let $\delta_i=\delta(w_i, w_{i+1})$ and set $\delta_{i_1}=\max_i \delta_i$. Let $W=\{w_i: i \le i_1\}$ and
$W'=\{w_i: i > i_1\}$. By the pigeonhole principle, either $|W| \ge n/2$ or
$|W'|\ge n/2$. Assume without loss of generality that $|W|\ge n/2$ and set $W_1=W$. Observe that by hypothesis and definition of $\delta_{i_1}$, for every $w_i, w_j \in W_1$, with $i<j$, we have
$$\phi(\delta(w_i, w_j),\delta_{i_1})=\phi(\delta(w_i, w_j), \delta(w_j, w_{i_1+1}))= \hbox{ red}.$$
Note that we obtain the same conclusion if $|W'| \geq n/2$ and $W_1=W'$ since
$$\phi(\delta_{i_1}, \delta(w_i,w_j)) = \phi(\delta(w_{i_1}, w_i), \delta(w_i, w_j)) = \textnormal{red}.$$
Now define $\delta_{i_2}=\max_{i < i_1} \delta_{i}$ and repeat the argument above to obtain $W_2$ with $|W_2|\ge n/4$ such that $\phi(\delta(w_i, w_j),\delta_{i_2})=$ red for  every $w_i, w_j \in W_2$, with $i<j$.
Continuing in this way, we obtain $\delta_{i_1}, \delta_{i_2}, \ldots, \delta_{i_{m }}$ for $m=\lfloor \log n \rfloor$, such that $\phi$ colors every pair in the set $\{\delta_{i_1}, \delta_{i_2}, \ldots, \delta_{i_{m}}\}$ red.  This contradicts Lemma \ref{offdiag}, and the statement follows.   \end{proof}

\begin{claim}\label{c2} There do not exist vertices $w_1<\cdots< w_{n^2}$ in $V$ such that every 4-tuple among them is blue under $\chi$ and for every $i<j<k$ with
$\delta(w_i,w_j) > \delta(w_j, w_k)$ we have
$\phi(\delta(w_i,w_j), \delta(w_j, w_k))=$ red.
\end{claim}

\begin{proof} Suppose for contradiction that these vertices $w_1<\cdots< w_{n^2}$ exist.  Let $\delta_i=\delta(w_i, w_{i+1})$ and set $\delta_{i_1}=\max_i \delta_i$.  Let $W=\{w_i: i \le i_1\}$ and
$W'=\{w_i: i > i_1\}$. Let us first suppose that
$|W'|\ge n$. Pick $w_i, w_j, w_k \in W'$ with $i<j<k$.  If $\delta(w_i,w_j) > \delta(w_j,w_k)$, then $\phi(\delta(w_i,w_j), \delta(w_j,w_k)) = $ red by assumption.  If $\delta(w_i,w_j) < \delta(w_j,w_k)$, then  consider the 4-tuple $w_{i_1}, w_i, w_j, w_k$.
Since this 4-tuple is blue under $\chi$, and both
$\phi(\delta(w_{i_1}, w_{i}), \delta(w_{i}, w_j))$ and
$\phi(\delta(w_{i_1}, w_{i}), \delta(w_{j}, w_k))$ are red,
$\phi(\delta(w_{i}, w_{j}), \delta(w_{j}, w_k))$ must also be red.
Now we may apply Claim \ref{c1} to $W'$ to obtain a contradiction.

We may therefore assume that $|W'| < n$ and hence $|W|\ge n^2-n \ge (n-1)^2$. We repeat the previous argument to $W$ to obtain $\delta_{i_2}$ and then $\delta_{i_3}, \ldots, \delta_{i_{n}}$, such that

$$\delta_{i_1} > \delta_{i_2} > \cdots > \delta_{i_{n}} \hspace{.5cm}\textnormal{and}\hspace{.5cm} i_1 > i_2 > \cdots > i_{n}.$$  Now consider the set $S=\{w_{i_1+1}, w_{i_2+1}, \ldots, w_{i_{n}+1}, w_{i_n}\}$, whose corresponding delta set is $A=\{\delta_{i_1}, \delta_{i_2}, \ldots, \delta_{i_{n}}\}$.  Then $A$ is an $n$-set that has the properties of Lemma~\ref{offdiag} part 3.  This implies that there are $j<k<l$ such that $\phi( \delta_{i_j}, \delta_{i_k})=\phi(\delta_{i_k}, \delta_{i_l})=$ blue and $\phi(\delta_{i_j}, \delta_{i_l})=$ red. Consequently,  $\chi(w_{i_j}, w_{i_k}, w_{i_l}, w_{i_l+1})=$ red, a contradiction.\end{proof}

By copying the proof above almost verbatim, we have the following.

\begin{claim}\label{c3} There do not exist vertices $w_1<\cdots< w_{n^2}$ in $V$ such that every 4-tuple among them is blue under $\chi$ and for every $i<j<k$ with
$\delta(w_i,w_j) < \delta(w_j, w_k)$ we have
$\phi(\delta(w_i,w_j), \delta(w_j, w_k))=$ red.\end{claim}

Now we are ready to show that there is no blue $K_m^{(4)}$ in coloring $\chi$, where $m = 2n^4$.  For the sake of contradiction, suppose we have vertices $v_1,\ldots, v_m \in V$ such that $v_1 < \cdots < v_m$, and $\chi$ colors every $4$-tuple in the set $\{v_1, \ldots, v_m\}$ blue.  Let $\delta_i = \delta(v_i,v_{i + 1})$ for $1\leq i \leq m - 1$.  Notice that by Observation~\ref{distinct} we have $\delta_i \neq \delta_j$ for $1 \leq i < j < m$.

Let $\delta^{\ast}_1 = \max\{\delta_1,\ldots, \delta_m\}$, where $\delta^{\ast}_1 = \delta(v_{i_1},v_{i_1 + 1})$.  Set $$V_1 = \{v_1,v_2,\ldots, v_{i_1 }\} \hspace{.5cm}\textnormal{and}\hspace{.5cm} V_2 = \{v_{i_1 + 1},v_{i_1 + 1}, \ldots, v_{m}\}.$$  Now we establish the following lemma.

\begin{lemma}\label{ststart3}
We have either $|V_1| < n^3 = m/2n$ or $|V_2| < n^3 = m/2n$.
\end{lemma}
\noindent \emph{Proof of Lemma \ref{ststart3}.}
 For the sake of contradiction, suppose $|V_1|,|V_2| \geq n^3$.  Recall that  $\delta^{\ast}_1 = \delta(v_{i_1},v_{i_1 + 1})$, $V_1 = \{v_1,v_2,\ldots, v_{i_1}\}$, $V_2 = \{v_{i_1 + 1},v_{i_1 + 1}, \ldots, v_{m}\}$, and set $A_1 = \{\delta_1,\ldots, \delta_{i_1 - 1}\}$ and $A_2 = \{\delta_{i_1 + 1},\ldots, \delta_{m-1}\}$.  For $i \in \{1,2\}$, let us partition $A_i = A_i^r\cup A_i^b$ where
$$A_i^r = \{\delta_j \in A_i: \phi(\delta^{\ast}_1,\delta_j) = \textnormal{ red}\}\qquad \hbox{ and } \qquad A^b_i = \{\delta_j \in A_i: \phi(\delta^{\ast}_1,\delta_j) = \textnormal{ blue}\}.$$

Let us first suppose that  $|A^b_i| \geq n$ for $i=1,2$.
Fix $\delta_{j_1} \in A_1^b$ and $\delta_{j_2} \in A_2^b$, and recall that $\delta_{j_1} = \delta(v_{j_1},v_{j_1 + 1})$ and $\delta_{j_2} = \delta(v_{j_2},v_{j_2 + 1})$.  By Observation~\ref{distinct},  $\delta_{j_1} \neq \delta_{j_2}$, and by Property II, we have $\delta(v_{j_1+1},v_{j_2}) = \delta^{\ast}_{1}$.  Since $\chi(v_{j_1},v_{j_1+1},v_{j_2},v_{j_2+1}) = $ blue, this implies that $\phi(\delta_{j_1},\delta_{j_2}) = $ blue.  Consequently, we have a monochromatic blue copy of $K_{n,n}$ in $A$ with respect to $\phi$, which contradicts Lemma~\ref{offdiag} part 2.

 Therefore we have $|A_1^b| \le n$ or $|A_2^b| \le n$.  Let us first suppose that $|A_1^b| \le n$. Since $|A_1|\ge n^3$, by the pigeonhole principle, there is a subset $R \subset A_1^r$ such that $R = \{\delta_j, \delta_{j + 1}, \ldots, \delta_{j + n^2 - 2}\}$, whose corresponding vertices are $U= \{v_j,v_{j+1},\ldots, v_{j + n^2 - 1}\}$.  For simplicity and without loss of generality, let us rename $U= \{u_1,\ldots, u_{n^2}\}$ and $\delta_i = \delta(u_i,u_{i+1})$ for $1 \leq i \leq n^2$.  Now notice that $\phi(\delta(u_i, u_j), \delta(u_j, u_k))=$ red for every $i<j<k$ with $\delta(u_i, u_j)> \delta(u_j, u_k)$.  Indeed, since $\delta(u_i, u_j),\delta(u_j, u_k) \in R$ we have $\phi(\delta^{\ast}_1, \delta(u_i, u_j))) = \phi(\delta^{\ast}_1, \delta(u_j, u_k))) = $ red.  Since $\chi(u_i, u_j, u_k, v_{i_1+1}) = $ blue, this implies that we must have  $\phi(\delta(u_i, u_j), \delta(u_j, u_k))=$ red by definition of $\chi$.  However, by Claim \ref{c2} we obtain a contradiction.

In the case that $|A_2^b| \le n$, a symmetric argument follows, where we apply Claim \ref{c3} instead of Claim \ref{c2} to obtain the contradiction. \qed

Now we can finish the argument that $\chi$ does not color every 4-tuple in the set $\{v_1,\ldots, v_m\}$ blue by copying the proof of Lemma~\ref{ststart2}.  In particular, we will obtain vertices $v_{j_1} < \cdots <  v_{j_{n + 1}} \in \{v_1,\ldots, v_m\}$ such that $\delta(v_{j_1},v_{j_2}), \delta(v_{j_2},v_{j_3}), \ldots, \delta(v_{j_n},v_{j_{n + 1}})$ forms a monotone sequence.  By Property IV and Lemma \ref{offdiag}, $\chi$ will color a 4-tuple in the set $\{v_{j_1}, \ldots, v_{j_{n + 1}}\}$ red.

\bigskip

{\bf Acknowledgment.}  We thank a referee for helpful comments.

\end{document}